\documentclass{article}
\usepackage[utf8]{inputenc}
\usepackage{amstext}
\usepackage{amsmath}
\usepackage{amssymb}
\usepackage{amsthm}
\usepackage{verbatim}
\usepackage[all]{xy}
\usepackage{perpage}
\MakePerPage{footnote}
\usepackage{color}
\usepackage{enumerate}
\usepackage{comment}
\usepackage{graphicx}
\usepackage{authblk}
\usepackage{appendix}
\usepackage[numbers]{natbib}
\usepackage{mathrsfs}
\usepackage{perpage}
\usepackage{adjustbox}
\usepackage{hyperref}
\hypersetup{colorlinks=true, citecolor=blue, linkcolor=blue}
\usepackage{footnote}
\usepackage{enumitem}

\newtheorem{definition}{Def\text{}inition}[section]
\newtheorem{theorem}[definition]{Theorem}
\newtheorem{example}[definition]{Example}

\newtheorem{lemma}[definition]{Lemma}

\newtheorem{corollary}[definition]{Corollary}

\usepackage{tikz-cd}

\begin{document}

\title{ \bf \large Variations of star selection principles on Hyperspaces\footnote{The author was supported for this research by Postdoctoral Fellowship Program at UNAM.}}
\author{ \small JAVIER CASAS-DE LA ROSA}
\date{}
\maketitle

\begin{abstract}
In this paper we define some combinatorial principles to characterize spaces $X$ whose hyperspace satisfies some variation of some classical star selection principle. Specifically, the variations characterized are the selective and absolute versions of the star selection principles for the Menger and Rothberger cases; also, the hyperspaces considered in these characterizations are $CL(X)$, $\mathbb{K}(X)$, $\mathbb{F}(X)$ and $\mathbb{CS}(X)$ in both cases, endowed with either the Fell topology or the Vietoris topology.
\end{abstract}



\section{Introduction and preliminaries}

Many branches from Selection Principles Theory have arisen after a systematic research made in \cite{MS1} by Scheepers. Nowadays, this theory has connections as well as applications to several areas of mathematics as General Topology, Function spaces, Hyperspaces, etc. As an example of it, in \cite{GK}, \cite{D},\cite{K'}  and \cite{Z}, the authors studied the fundamental problem on Hyperspaces which consists in establishing dualities between some topological properties or, as it is in this case, between some classical selection principles under different hyperspaces topologies. In other words, given two selection principles $\mathcal{P}$ and $\mathcal{Q}$, one must determine if it holds that a topological space $X$ satisfies the principle $\mathcal{P}$ if and only if its hyperspace satisfies the principle $\mathcal{Q}$. Hence, this duality problem can be viewed as a method to characterize some classical selection principles  of $X$ with selection principles on different hyperspaces of $X$, using several well-known topologies.

On the other hand, several versions of the original selection principles have been defined from its beginning. One of the most important versions that its investigation has rapidly increased is the star versions of the classical selection principles. These star versions were defined by Ko\v{c}inac in \cite{K} and they gave rise to the star selection principles theory (see \cite{K_survey} and \cite{K_Axioms}). 

The study of the duality problem involving some star selection principles on hyperspaces initiated very recently in \cite{xiy} and continued in \cite{CMR}, \cite{CRT2}, \cite{CRT} and \cite{DRT} have dealt with classical star selection principles only. In particular, in these works several combinatorial principles have been defined to establish characterizations for the (strongly) star-Menger property and the (strongly) star-Rothberger property on several hyperspaces under different topologies. In this paper we continue this line of investigation for some variations of some classical star selection principles. In Section \ref{Vietoris section}, the selective and absolute versions of Menger-type and Rothberger-type star selection principle on several hyperspaces with the Vietoris topology are characterized. In Section \ref{Fell section}, analogous characterizations are given for several hyperspaces with the Fell topology.

\subsection{Hyperspaces}

All spaces are assumed to be Hausdorff noncompact and, even, nonparacompact. Given a topological space $X$ , the hyperspace of $X$, denoted by $CL(X)$, is the set of all nonempty closed subsets of $X$. By $\mathbb{K}(X)$ ($\mathbb{F}(X)$), we denote the family of all nonempty compact (all nonempty finite) subsets of $X$. Also, by $\mathbb{CS}(X)$ we denote the family of all convergent sequences of $X$.\\ 

We denote by $\omega$ the first infinite cardinal and for a set $A$, $[A]^{< \omega}$ denotes the set of all finite subsets of $A$. For a subset $U\subseteq X$ and a family $\mathcal{U}$ of subsets of $X$, we write:
\begin{center}
\begin{tabular}{lcl}
 $U^-$ & = & $\{A\in CL(X):A\cap U\neq\emptyset\}$;\\
 $U^+$ & = & $\{A\in CL(X):A\subseteq U\}$;\\
 $U^c$ & = & $X\backslash U$;\\
 $\mathcal{U}^c$ & = & $\{U^c:U\in \mathcal{U}\}$;\\
 $\mathcal{U}^-$ & = & $\{U^-:U\in \mathcal{U}\}$;\\
 $\mathcal{U}^+$ & = & $\{U^+:U\in\mathcal{U}\}$.
\end{tabular} 
\end{center}

In the literature there are many topologies that can be defined on $CL(X)$ or on a subcollection of it. In this paper we will consider two well-known topologies, the Vietoris topology, denoted by $\mathbf{V}$, and the Fell topology, denoted by $\mathbf{F}$.\\
    
\noindent A basic open subset of the Fell topology is of the form:$$(\bigcap_{i=1}^nV_i^-)\cap(K^c)^+,$$ where $V_1,\ldots, V_n$ are open subsets of $X$ and $K$ is a compact subset of $X$.\\

\noindent A basic open subset of the Vietoris topology is a set of the form:$$\langle U_1,\ldots, U_n\rangle=\{A\in CL(X): A\subseteq\bigcup_{i=1}^nU_i, A\cap U_i\neq\emptyset\text{ for each }i\leq n\}$$ where $U_1,\ldots, U_n$ are open subsets of $X$, $n\in\omega$.

\subsection{Variations of the classical star selection principles}

For a set $A\subseteq X$ and a collection $\mathcal{U}$ of subsets of $X$, the star of $A$ with respect to $\mathcal{U}$, denoted by $St(A,\mathcal{U})$, is the set $\bigcup\{U\in\mathcal{U}:U\cap A\neq\emptyset\}$; for $A=\{x\}$ with $x\in X$, we write $St(x,\mathcal{U})$ instead of $St(\{x\},\mathcal{U})$.

In recent years, different versions of the classical star selection principles have been defined and studied in several articles (see for instance \cite{BCKM}, \cite{CDK} and \cite{CG}). Some of these new star selection principles involve dense subsets of the space and they are called as the absolute and selective versions of classical star selection principles (see \cite{K_survey} for more information about the absolute versions and see \cite{CG} for an overview about the selective versions). The selective versions are stronger than the absolute versions and the absolute versions are stronger than the classical star selection principles\footnote{In \cite{CDK}, the authors introduced the absolute versions of classical star selection principles in a general form with a different notation than the one used in \cite{CG}, where it was defined the absolute and selective versions of star selection principles, also given in a general form.}. The following properties are the absolute versions of the strongly star selection principles for the Menger and Rothberger cases.

\begin{definition}[\cite{CDK}]
We say that a space $X$ is:
\begin{enumerate}
    \item absolutely strongly star-Menger ($aSSM$) if for each sequence $\{\mathcal{U}_n:n\in\omega\}$ of open covers of $X$ and each dense subset $D$ of $X$, there is a sequence $\{F_n:n\in\omega\}$ of finite sets such that $F_n\subseteq D$, $n\in\omega$, and $\{St(F_n,\mathcal{U}_n):n\in\omega\}$ is an open cover of $X$.
    \item absolutely strongly star-Rothberger ($aSSR$) if for each sequence $\{\mathcal{U}_n:n\in\omega\}$ of open covers of $X$ and each dense subset $D$ of $X$, there is a sequence $\{x_n:n\in\omega\}$ of points of $X$ such that $x_n\in D$, $n\in\omega$, and $\{St(x_n,\mathcal{U}_n):n\in\omega\}$ is an open cover of $X$.
\end{enumerate}
\end{definition}

Now, we recall the selective versions of the strongly star selection principles for the Menger and Rothberger cases\footnote{The Hurewicz case and some other interesting properties are also given in \cite{CG}}.

\begin{definition}[\cite{CG}]
We say that a space $X$ is:
\begin{enumerate}
    \item selectively strongly star-Menger ($selSSM$) if for each sequence $\{\mathcal{U}_n:n\in\omega\}$ of open covers of $X$ and each sequence $\{D_n:n\in\omega\}$ of dense sets of $X$, there exists a sequence $\{F_n:n\in\omega\}$ of finite sets such that $F_n\subseteq D_n$, $n\in\omega$, and $\{St(F_n,\mathcal{U}_n):n\in\omega\}$ is an open cover of $X$.
    \item selectively strongly star-Rothberger ($selSSR$) if for each sequence $\{\mathcal{U}_n:n\in\omega\}$ of open covers of $X$ and each sequence $\{D_n:n\in\omega\}$ of dense sets of $X$, there exists a sequence $\{x_n:n\in\omega\}$ of points of $X$ such that $x_n\in D_n$, $n\in\omega$, and $\{St(x_n,\mathcal{U}_n):n\in\omega\}$ is an open cover of $X$.
\end{enumerate}
\end{definition}

\section{Absolute and selective versions on hyperspaces with the Vietoris topology}\label{Vietoris section}

In this section we present some technical principles that are useful to characterize some variations of classical star selection principles ($selSSM$, $selSSR$, $aSSM$ and $aSSR$) on several hiperspaces with the Vietoris topology. To that end, we recall some notation and useful definitions to establish these characterizations.\\

\noindent Using the notation of \cite{Z}, $\zeta$ denotes the family:
\[
\zeta =\{(V_1,\ldots,V_n):  V_1,\ldots, V_n\text{  are }\text{open subsets of }X,\ \ n\in\mathbb{N}\}.
\]

\noindent In \cite{Z}, Li defined the notion of $\pi_V$-networks in a space $X$ as follows: 

\medskip

A family $\zeta$ is called a \emph{$\pi_V$-network of $X$} if for each open subset $U$ of $X$, with $U\neq X$, there exist  $(V_1,\ldots, V_n)\in\zeta$ and a finite set $F$ with $F\cap V_i\neq\emptyset$ $(1\leq i\leq n)$ such that $\bigcap_{i=1}^n V_i^c\subseteq U$ and $F\cap U=\emptyset$. The collection of $\pi_V$-networks of a space $X$ is denoted by $\Pi_V$.

\medskip

Henceforth, $\Delta$ will denote a subset of $CL(X)$ which is closed under finite unions and contains all singletons. Using a family $\Delta$, in \cite{DRT}, it was defined a modification of $\pi_V$-network which is called as $\pi_V(\Delta)$-network of $X$.

\begin{definition}[\cite{DRT}]\label{def:pi_Gama_network}
A family $\zeta$ is called a $\pi_V(\Delta)$-network of $X$, if for each $U\in \Delta^c$, there exist a $( V_1,\ldots, V_n)\in\zeta$ and  $F\in [X]^{<\omega}$ with $F\cap V_i\neq\emptyset$ $(1\leq i\leq n)$ such that $\bigcap_{i=1}^n V_i^c\subset U$ and $F\cap U=\emptyset$. The collection of all $\pi_V(\Delta)$-networks of $X$ is denoted by $\Pi_V(\Delta)$. 
\end{definition}

As pointed out in \cite{DRT}, if we consider $\Delta$ to be $CL(X)$, then the collections $\Pi_V(\Delta)$ and $\Pi_V$ coincide. But this fact not need be true in general. In \cite{DRT}, the authors gave an example of a $\pi_V(\Delta)$-network that is not a $\pi_V$-network (on a metrizable space $X$ and a specific family $\Delta$). Here we present another example in a non-metrizable space\footnote{This space was also used in \cite{xiy} to show that $\pi_F(\Delta)$-network and $\pi_F$-network are different notions; see Section \ref{Fell section} for definitions.}.

\begin{example}
There exists a family $\zeta$ on a certain space $X$ and there exists a family $\Delta\subseteq CL(X)$ such that $\zeta$ is a $\pi_V(\Delta)$-network of $X$ but is not a $\pi_V$-network of $X$.
\end{example}
\begin{proof}
Let $\mathcal{A}$ be an uncountable almost disjoint family on $\omega$ with $\omega=\bigcup\mathcal{A}$. We consider the Mr\'owka-Isbell space (see \cite{PM}), $X=\Psi(\mathcal{A})$ and $\Delta=\mathbb{K}(X)$. Let
\begin{center}
    $\zeta=\{(V_1,\ldots,V_n): V_i=\{A_i\}\cup (A_i\setminus D_i)$ $(1\leq i\leq n)$ where $A_1,\ldots, A_n\in\mathcal{A}$, $D_1\ldots, D_n\in [\omega]^{<\omega}$, $n\in\mathbb{N}\}$
\end{center}
Note that $\zeta$ is properly defined, that is, for each $i\in\{1,\dots,n\}$, $\{A_i\}\cup (A_i\setminus D_i)$ is an open set of $X$.\\

\noindent\emph{Claim 1:} $\zeta$ is a $\pi_V(\mathbb{K}(X))$-network of $X$.\\
Let $U\in[\mathbb{K}(X)]^c$. Suppose that $U=X\setminus K_0$ for some $K_0\in\mathbb{K}(X)$. We consider the following two cases:\\
\emph{Case I:} $K_0$ is infinite.\\
Note that in this case, $K_0\cap\mathcal{A}$ is a nonempty finite set. Moreover, the set $(K_0\cap\omega)\setminus(\bigcup\{B:B\in K_0\cap\mathcal{A}\})$ is finite. For each $m\in (K_0\cap\omega)\setminus(\bigcup\{B:B\in K_0\cap\mathcal{A}\})$, let $B_m\in\mathcal{A}$ be such that $m\in B_m$. The collection $\mathcal{B}=(K_0\cap\mathcal{A})\cup\{B_m:m\in (K_0\cap\omega)\setminus(\bigcup\{B:B\in K_0\cap\mathcal{A}\})\}$ is  finite. We enumerate this collection as $\mathcal{B}=\{A_1,\ldots, A_n\}$. We define, for each $i\in\{1,\ldots,n\}$, $V_i=\{A_i\}\cup A_i$. Then, each $V_i$ is an open set of $X$ and therefore, $(V_1,\dots, V_n)\in\zeta$. Note that $K_0\subseteq\bigcup_{i=1}^nV_i$. Thus, $\bigcap_{i=1}^nV_i^c\subseteq U$. Finally, we let $F=(K_0\cap\mathcal{A})\cup[(K_0\cap\omega)\setminus(\bigcup\{B:B\in K_0\cap\mathcal{A}\})]$. Then, $F$ is a finite set of $X$. Moreover, note that $F\cap V_i\neq\emptyset$ $(1\leq i\leq n)$ and $F\cap U=\emptyset$. Therefore, $\zeta$ is a $\pi_V(\mathbb{K}(X))$-network of $X$.\\
\emph{Case II:} $K_0$ is finite.\\
For each $m\in K_0\cap\omega$ (if $K_0\cap\omega\neq\emptyset$), let $B_m\in\mathcal{A}$ such that $m\in B_m$. Then, the collection $\mathcal{B}=(K_0\cap\mathcal{A})\cup\{B_m:m\in K_0\cap\omega\}$ is finite. We enumerate this collection as $\mathcal{B}=\{A_1,\ldots, A_n\}$. We define, for each $i\in\{1,\ldots,n\}$, $V_i=\{A_i\}\cup A_i$. Then, each $V_i$ is an open set of $X$ and therefore, $(V_1,\dots, V_n)\in\zeta$. Note that $K_0\subseteq\bigcup_{i=1}^nV_i$. Hence, $\bigcap_{i=1}^nV_i^c\subseteq U$. Let $F=K_0$. Then, $F$ is a finite set of $X$ such that $F\cap V_i\neq\emptyset$ $(1\leq i\leq n)$ and $F\cap U=\emptyset$. Therefore, $\zeta$ is a $\pi_V(\mathbb{K}(X))$-network of $X$.\\ 

\noindent\emph{Claim 2:} $\zeta$ is not a $\pi_V$-network of $X$.\\
Let $U=\omega$. Then, $U$ is an open set of $X$ with $U\neq X$. Let $(V_1,\dots,V_n)$ be any element of $\zeta$, where $V_i=\{A_i\}\cup (A_i\setminus D_i)$ $(1\leq i\leq n)$ for some $A_1,\ldots, A_n\in\mathcal{A}$, $D_1\ldots, D_n\in [\omega]^{<\omega}$, $n\in\mathbb{N}$. Then, $\bigcap_{i=1}^n[\{A_i\}\cup (A_i\setminus D_i)]^c\nsubseteq U$. Since this fact holds for any element of $\zeta$, we conclude that $\zeta$ is not a $\pi_V$-network of $X$.
\end{proof}

From now on, if $J_n$ is an element in $\Pi_V(\Delta)$, we put:   
\[
J_n=\{(V_{1,s}^n,\ldots,V_{m_s,s}^n):s\in S_n\}.
\]

Another notion, defined also in \cite{DRT}, that involves a family $\Delta$ is the following:

\begin{definition}[\cite{DRT}]\label{definition c_V(Delta)-cover}
Let $(X,\tau)$ be a topological space. A family $\mathcal{U}\subseteq \Delta^c$ is called a \emph{$c_V(\Delta)$-cover of $X$}, if for any open subsets $V_1,\ldots , V_m$ of $X$, there exists $U\in\mathcal{U}$ and $F\in[X]^{<\omega}$ such that for each $i\in\{1,\ldots ,m\}$, $F\cap V_i\neq\emptyset$, $\bigcap_{i=1}^m V_i^c\subseteq U$ and  $F\cap U = \emptyset$. The family of all $c_V(\Delta)$-covers of a space $X$ is denoted by $\mathbb{C}_V(\Delta)$.
\end{definition}


The following two lemmas will be useful in the proofs of next results in this section and the proofs of these lemmas can be easily obtained; we refer the reader to \cite{DRT} for details. The first lemma says how $c_V(\Delta)$-covers on a space $X$ can be viewed as dense subspaces of certain hyperspaces of $X$.

\begin{lemma}[\cite{DRT}]\label{cvdeltas are dense sets}
Let $X$ be a topological space and $\mathcal{U}\subseteq \Delta^c$. Then $\mathcal{U}$ is a $c_V(\Delta)$-cover of $X$ if and only if $\mathcal{U}^c$ is a dense subset of $(\Delta,\mathbf{V})$. 
\end{lemma}

The next lemma says how $\pi_V(\Delta)$-networks of a space $X$ can be interpreted as open covers of certain hyperspaces of $X$.

\begin{lemma}[\cite{DRT}]\label{piVdeltas are open covers}
Let $X$ be a topological space and $\zeta = \{ (V_1, \ldots, V_n): V_1, \ldots, V_n$ $\text{ are open subsets of } X, n \in \omega \}.$ Then $\zeta$ is a $\pi_V(\Delta)$-network of $X$ if and only if the collection $\mathcal{U}=\{\langle V_1,\ldots,V_n\rangle: (V_1,\ldots,V_n)\in\zeta\}$ is an open cover of $(\Delta,\mathbf{V})$. 
\end{lemma}

The following selection principle will help us to characterize the selectively strongly star-Menger property on hyperspaces with the Vietoris topology.

\begin{definition}\label{SVMdefinition}
Let $X$ be a topological space. We define:

\medskip

\noindent $\mathbf{SV}_{M}(\Pi_V(\Delta),\Pi_V(\Delta))$: For each sequence $\{J_n:n\in\mathbb{N}\}\subseteq\Pi_V(\Delta)$ and each sequence $\{\mathcal{C}_n:n\in\mathbb{N}\}\subseteq \mathbb{C}_V(\Delta)$, there are finite subsets $\mathcal{V}_n\subseteq\mathcal{C}_n$, $n\in\mathbb{N}$, such that $\mathcal{J}=\bigcup_{n\in\mathbb{N}}\{  (V_{1,s}^n,\ldots,V_{m_s,s}^n)\in J_n:\text{ there exists } U_n\in \mathcal{V}_n \text{ such that }$ $\bigcap_{i=1}^{m_s}(V_{i,s}^n)^c \subseteq U_n, V_{i,s}^n \nsubseteq U_n \ (1\leq i\leq m_s)\}$ is an element of $\Pi_V(\Delta)$.
\end{definition}

\begin{theorem}\label{VselSSMtheorem}
Given a topological space $X$, the following conditions are equivalent:
\begin{enumerate}
\item[(1)] $(\Delta,\mathbf{V})$ is selSSM;
\item[(2)] $X$ satisfies $\mathbf{SV}_{M}(\Pi_V(\Delta),\Pi_V(\Delta))$.
\end{enumerate}
\end{theorem}
\begin{proof}
$(1) \Rightarrow (2)$: Let $\{J_n:n\in\mathbb{N}\}$ be a sequence of $\pi_V(\Delta)$-networks of $X$ and let $\{\mathcal{C}_n: n\in\mathbb{N}\}$ be a sequence of $c_V(\Delta)$-covers of $X$. By Lemma \ref{cvdeltas are dense sets}, the collections $\mathcal{D}_n=\mathcal{C}_n^c$ are dense subsets of $(\Delta,\mathbf{V})$, for each $n\in\mathbb{N}$. Furthermore, if we put, for each $n\in\omega$, $J_n=\{(V_{1,s}^n,\ldots,V_{m_s,s}^n):s\in S_n\}$, then by Lemma \ref{piVdeltas are open covers}, the collections $\mathcal{U}_n=\{\langle V_{1,s}^n,\ldots,V_{m_s,s}^n\rangle:s\in S_n\}$ are open covers of $(\Delta,\mathbf{V})$, for each $n\in\mathbb{N}$. 

Now, applying (1) to the sequence $\{\mathcal{U}_n: n\in\mathbb{N}\}$ and the sequence $\{\mathcal{D}_n:n\in\mathbb{N}\}$, there is a sequence $\{\mathcal{A}_n:n\in\mathbb{N}\}$ of finite sets such that, for each $n\in \mathbb{N}$, $\mathcal{A}_n\subseteq \mathcal{D}_n$ and the collection $\{St(\mathcal{A}_n,\mathcal{U}_n):n\in\mathbb{N}\}$ is an open cover of $(\Delta,\mathbf{V})$. We put, for each $n\in\mathbb{N}$, $\mathcal{V}_n=\{A^c: A\in \mathcal{A}_n\}$. Then, for each $n\in\mathbb{N}$, $\mathcal{V}_n$ is a finite subset of $\mathcal{C}_n$.

Let us show that the collection $\mathcal{J}=\bigcup_{n\in\mathbb{N}}\{(V_{1,s}^n,\ldots,V_{m_s,s}^n)\in J_n: \exists \ U_n\in \mathcal{V}_n \text{ such that }\bigcap_{i=1}^{m_s}(V_{i,s}^n)^c \subseteq U_n, \  V_{i,s}^n\nsubseteq U_n \ (1\leq i\leq m_s)\}$ is a $\pi_V(\Delta)$-network of $X$.

Let $U\in \Delta^c$. Then $U^c\in \Delta$ and therefore, there exists $n_0\in\omega$ such that $U^c\in St(\mathcal{A}_{n_0}, \mathcal{U}_{n_0})$. 
Then, there are $\langle V_{1,s_0}^{n_0},\ldots,V_{m_{s_0},s_0}^{n_0}\rangle\in \mathcal{U}_{n_0}$ and $A_{n_0}\in \mathcal{A}_{n_0}$ so that $\{U^c, A_{n_0}\}\subseteq \langle V_{1,s_0}^{n_0},\ldots,V_{m_{s_0},s_0}^{n_0}\rangle$. Let $U_{n_0}=A^c_{n_0}$. Then $U_{n_0}\in\mathcal{V}_{n_0}$. Since $A_{n_0}$ belongs to $\langle V_{1,s_0}^{n_0},\ldots,V_{m_{s_0},s_0}^{n_0}\rangle$, it follows that $\bigcap_{i=1}^{m_{s_0}}(V_{i,s_0}^{n_0})^c\subseteq U_{n_0},\ \ V_{i,s_0}^{n_0}\nsubseteq U_{n_0} (1\leq i\leq m_{s_0})$; hence $(V_{1,s_0}^{n_0},\ldots,V_{m_{s_0},s_0}^{n_0})\in\mathcal{J}$. On the other hand, using the fact that $U^c$ also belongs to $\langle V_{1,s_0}^{n_0},\ldots,V_{m_{s_0},s_0}^{n_0}\rangle$, we can take, for each $i\in\{1,\ldots, m_{s_0}\}$, $x_i\in U^c\cap V_{i,s_0}^{n_0}$. We put $F=\{x_i : i\in\{1,\ldots ,m_{s_0}\}\}$. Hence, $F\in [X]^{<\omega}$ with $F\cap V_{i,s_0}^{n_0}\neq\emptyset$ and $F\cap U=\emptyset$. 
Moreover, since $U^c\subseteq \bigcup_{i=1}^{m_{s_0}}V_{i,s_0}^{n_0}$, then we obtain that $\bigcap_{i=1}^{m_{s_0}}(V_{i,s_0}^{n_0})^c\subseteq U$. We conclude that $\mathcal{J}\in\Pi_V(\Delta)$.

$(2) \Rightarrow (1)$: Let $\{\mathcal{U}_n:n\in\mathbb{N}\}$ be a sequence of open covers of $(\Delta,\mathbf{V})$ and let $\{\mathcal{D}_n:n\in\mathbb{N}\}$ be a sequence of dense subsets of $(\Delta,\mathbf{V})$. We can assume that each open cover $\mathcal{U}_n$ consists of basic open subsets in $CL(X)$. Thus, put for each $n\in\mathbb{N}$, $\mathcal{U}_n=\{\langle V_{1,s}^n,\ldots,V_{m_s,s}^n\rangle:s\in S_n\}$, where  $V_{i,s}^n$ is an open subset of $X$, for every $n\in\mathbb{N}$, $s\in S_n$ and $i\in\{1,\ldots ,m_s\}$.
Let $J_n=\{(V_{1,s}^n,\ldots,V_{m_s,s}^n):s\in S_n\}$ for every $n\in\mathbb{N}$. By Lemma \ref{piVdeltas are open covers}, note that each $J_n$ is a $\pi_V(\Delta)$-network of $X$. On the other hand, for each $n\in\mathbb{N}$, let $\mathcal{C}_n = \mathcal{D}_n^c$. Thus, by Lemma \ref{cvdeltas are dense sets}, each $\mathcal{C}_n$ is a $c_V(\Delta)$-cover of $X$.

We apply $(2)$ to the sequence of $\pi_V(\Delta)$-networks $\{J_n:n\in \mathbb{N}\}$ and the sequence of $c_V(\Delta)$-covers $\{\mathcal{C}_n:n\in\mathbb{N}\}$ to obtain a sequence $\{\mathcal{V}_n:n\in\mathbb{N}\}$ such that, for each $n\in\mathbb{N}$, $\mathcal{V}_n\in[\mathcal{C}_n]^{<\omega}$, and the collection  $\mathcal{J}=\bigcup_{n\in\mathbb{N}}\{(V_{1,s}^n,\ldots,V_{m_s,s}^n)\in J_n:\text{ there exists } U_n\in \mathcal{V}_n \text{ such that }\bigcap_{i=1}^{m_s}(V_{i,s}^n)^c\subseteq U_n, \  V_{i,s}^n\nsubseteq U_n \ \ (1\leq i\leq m_s)\}$ is a $\pi_V(\Delta)$-network of $X$. For each $n\in\omega$, we define $\mathcal{A}_n=\mathcal{V}^c_{n}$. It follows that $\mathcal{A}_n\in [\mathcal{D}_n]^{<\omega}$, for each $n\in\mathbb{N}$.

Let us show that the collection $\{St(\mathcal{A}_n,\mathcal{U}_n):n\in\mathbb{N}\}$ is an open cover of $(\Delta,\mathbf{V})$. Let $A\in\Delta$. Since $\mathcal{J}$ is a $\pi_V(\Delta)$-network of $X$ and $A^c\in\Delta^c$, there exist $(V_{1,s_0}^{n_0},\ldots, V_{m_{s_0},s_0}^{n_0})\in\mathcal{J}$ (for some $n_0\in\mathbb{N}$ and some $s_0\in S_{n_0}$) and a finite set $F\subseteq X$ such that for every $i\in\{1,\ldots ,m_{s_0}\}$, $F\cap V_{i,s_0}^{n_0}\neq\emptyset$, \   $\bigcap_{i=1}^{m_{s_0}}(V_{i,s_0}^{n_0})^c\subseteq A^c$ and $F\cap A^c=\emptyset$. Since $(V_{1,s_0}^{n_0},\ldots, V_{m_{s_0},s_0}^{n_0})\in\mathcal{J}$, there is $U_{n_0}\in \mathcal{V}_{n_0}$ such that
 $\bigcap_{i=1}^{m_{s_0}}(V_{i,s_0}^{n_0})^c\subseteq U_{n_0}$ and for each $i\in\{1,\ldots ,m_{s_0}\}$, $V_{i,s_0}^{n_0}\nsubseteq U_{n_0}$. It means that $\{A, A_{n_0}\}\subseteq \langle V_{1,s_0}^{n_0},\ldots, V_{m_{s_0},s_0}^{n_0}\rangle\in \mathcal{U}_{n_0}$, where $A_{n_0}=U^c_{n_0}$. Since $A_{n_0}$ is an element of $\mathcal{A}_{n_0}$, we obtain that $A\in St(\mathcal{A}_{n_0},\mathcal{U}_{n_0})$. This shows that the collection $\{St(\mathcal{A}_n,\mathcal{U}_n):n\in\mathbb{N}\}$ is an open cover of $(\Delta,\mathbf{V})$.
\end{proof}

We obtain the following particular cases by taking different choices of our family $\Delta$.

        \begin{corollary}\label{VselSSMcorollary}
    Let $X$ be a topological space. Then:
    \begin{enumerate}
        \item $(CL(X),\mathbf{V})$ is selSSM if and only if $X$ satisfies $\mathbf{SV}_M(\Pi_V,\Pi_V)$;
        \item $(\mathbb{K}(X),\mathbf{V})$ is selSSM if and only if $X$ satisfies $\mathbf{SV}_M(\Pi_V(\mathbb{K}(X)),\Pi_V(\mathbb{K}(X)))$;
        \item $(\mathbb{CS}(X),\mathbf{V})$ is selSSM if and only if $X$ satisfies $\mathbf{SV}_M(\Pi_V(\mathbb{CS}(X)),\Pi_V(\mathbb{CS}(X)))$;
        \item $(\mathbb{F}(X),\mathbf{V})$ is selSSM if and only if $X$ satisfies $\mathbf{SV}_M(\Pi_V(\mathbb{F}(X)),\Pi_V(\mathbb{F}(X)))$.
    \end{enumerate}
    \end{corollary}
    
Let us now define another selection principle to characterize the selectively strongly star-Rothberger property on hyperspaces with the Vietoris topology.

\begin{definition}\label{SVRdefinition}
Let $X$ be a topological space. We define:

\medskip

\noindent $\mathbf{SV}_{R}(\Pi_V(\Delta),\Pi_V(\Delta))$: For each sequence $\{J_n:n\in\mathbb{N}\}\subseteq\Pi_V(\Delta)$ and each sequence $\{\mathcal{C}_n:n\in\mathbb{N}\}\subseteq \mathbb{C}_V(\Delta)$, there is a sequence $\{C_n:n\in\mathbb{N}\}$ with $C_n\in\mathcal{C}_n$, $n\in\mathbb{N}$, such that $\mathcal{J}=\bigcup_{n\in\mathbb{N}}\{  (V_{1,s}^n,\ldots,V_{m_s,s}^n)\in J_n: \bigcap_{i=1}^{m_s}(V_{i,s}^n)^c\subseteq C_n,$ $V_{i,s}^n\nsubseteq C_n \  (1\leq i\leq m_s)\}$
is an element of $\Pi_V(\Delta)$.
\end{definition}

\begin{theorem}\label{VselSSRtheorem}
Given a topological space $X$, the following conditions are equivalent:
\begin{enumerate}
\item[(1)] $(\Delta,\mathbf{V})$ is selSSR;
\item[(2)] $X$ satisfies $\mathbf{SV}_{R}(\Pi_V(\Delta),\Pi_V(\Delta))$.
\end{enumerate}
\end{theorem}
\begin{proof}
$(1) \Rightarrow (2)$: By mimicking the first part of $(1) \Rightarrow (2)$ in the proof of Theorem \ref{VselSSMtheorem}, we can obtain a sequence $\{D_n:n\in\mathbb{N}\}$ of points in $(\Delta,\mathbf{V})$ such that, for each $n\in \mathbb{N}$, $D_n\in \mathcal{D}_n$ and the collection $\{St(D_n,\mathcal{U}_n):n\in\mathbb{N}\}$ is an open cover of $(\Delta,\mathbf{V})$. We put, for each $n\in\mathbb{N}$, $C_n=D_n^c$. Then, $C_n\in\mathcal{C}_n$, for each $n\in\mathbb{N}$.

Let us show that, taking the sequence $\{C_n:n\in\mathbb{N}\}$, the collection $\mathcal{J}=\bigcup_{n\in\mathbb{N}}\{(V_{1,s}^n,\ldots,V_{m_s,s}^n)\in J_n: \bigcap_{i=1}^{m_s}(V_{i,s}^n)^c \subseteq C_n, \  V_{i,s}^n\nsubseteq C_n \ (1\leq i\leq m_s)\}$ is a $\pi_V(\Delta)$-network of $X$.

Let $U\in \Delta^c$. Then there exists $n_0\in\omega$ such that $U^c\in St(D_{n_0}, \mathcal{U}_{n_0})$. 
Thus, $U^c, D_{n_0}\in \langle V_{1,s_0}^{n_0},\ldots,V_{m_{s_0},s_0}^{n_0}\rangle$ for some $\langle V_{1,s_0}^{n_0},\ldots,V_{m_{s_0},s_0}^{n_0}\rangle\in \mathcal{U}_{n_0}$. Since $C_{n_0}=D_{n_0}^c$ and $D_{n_0}\in \langle V_{1,s_0}^{n_0},\ldots,V_{m_{s_0},s_0}^{n_0}\rangle$, it follows that $\bigcap_{i=1}^{m_{s_0}}(V_{i,s_0}^{n_0})^c\subseteq C_{n_0}$ and $ \ V_{i,s_0}^{n_0}\nsubseteq C_{n_0} (1\leq i\leq m_{s_0})$; hence $(V_{1,s_0}^{n_0},\ldots,V_{m_{s_0},s_0}^{n_0})\in\mathcal{J}$. In addition, using that $U^c$ is an element of $\langle V_{1,s_0}^{n_0},\ldots,V_{m_{s_0},s_0}^{n_0}\rangle$, we obtain that $\bigcap_{i=1}^{m_{s_0}}(V_{i,s_0}^{n_0})^c\subseteq U$ and also, we can define a finite subset $F$ of $X$ such that $F\cap V_{i,s_0}^{n_0}\neq\emptyset$ and $F\cap U=\emptyset$. We conclude that $\mathcal{J}\in\Pi_V(\Delta)$.\\

$(2) \Rightarrow (1)$: In the same way as in the first part of $(2) \Rightarrow (1)$ in the proof of Theorem \ref{VselSSMtheorem}, we obtain a sequence $\{C_n:n\in\mathbb{N}\}$ such that, for each $n\in\mathbb{N}$, $C_n\in\mathcal{C}_n$, and the collection  $\mathcal{J}=\bigcup_{n\in\mathbb{N}}\{(V_{1,s}^n,\ldots,V_{m_s,s}^n)\in J_n: \bigcap_{i=1}^{m_s}(V_{i,s}^n)^c\subseteq C_n, \  V_{i,s}^n\nsubseteq C_n \ \ (1\leq i\leq m_s)\}$ is a $\pi_V(\Delta)$-network of $X$. We put $D_n= C^c_{n}$ for each $n\in\omega$. Then, we have that $\{D_n:n\in\mathbb{N}\}$ is a sequence of points in $(\Delta,\mathbf{V})$ with $D_n\in\mathcal{D}_n$, for each $n\in\mathbb{N}$.

Let us show that the collection $\{St(D_n, \mathcal{U}_n):n\in\mathbb{N}\}$ is an open cover of $(\Delta,\mathbf{V})$. Let $A\in\Delta$. There exists $(V_{1,s_0}^{n_0},\ldots, V_{m_{s_0},s_0}^{n_0})\in\mathcal{J}$ (for some $n_0\in\mathbb{N}$ and some $s_0\in S_{n_0}$) and a finite set $F\subseteq X$ such that for every $i\in\{1,\ldots ,m_{s_0}\}$, $F\cap V_{i,s_0}^{n_0}\neq\emptyset$, \   $\bigcap_{i=1}^{m_{s_0}}(V_{i,s_0}^{n_0})^c\subseteq A^c$ and $F\cap A^c=\emptyset$. Using the fact that $(V_{1,s_0}^{n_0},\ldots, V_{m_{s_0},s_0}^{n_0})$ is an element of $\mathcal{J}$, we get that $\bigcap_{i=1}^{m_{s_0}}(V_{i,s_0}^{n_0})^c\subseteq C_{n_0}$ and for each $i\in\{1,\ldots ,m_{s_0}\}$, $V_{i,s_0}^{n_0}\nsubseteq C_{n_0}$. Thus, $A, D_{n_0}$ are in $\langle V_{1,s_0}^{n_0},\ldots, V_{m_{s_0},s_0}^{n_0}\rangle$ which is an element of $ \mathcal{U}_{n_0}$. In other words, $A\in St(D_{n_0},\mathcal{U}_{n_0})$. This shows that the collection $\{St(D_n,\mathcal{U}_n):n\in\mathbb{N}\}$ is an open cover of $(\Delta,\mathbf{V})$.
\end{proof}

We obtain the following particular cases by taking different choices of our family $\Delta$.

\begin{corollary}\label{VselSSRcorollary}
Let $X$ be a topological space. Then:
\begin{enumerate}
    \item $(CL(X),\mathbf{V})$ is selSSR if and only if $X$ satisfies $\mathbf{SV}_R(\Pi_V,\Pi_V)$;
    \item $(\mathbb{K}(X),\mathbf{V})$ is selSSR if and only if $X$ satisfies $\mathbf{SV}_R(\Pi_V(\mathbb{K}(X)),\Pi_V(\mathbb{K}(X)))$;
    \item $(\mathbb{CS}(X),\mathbf{V})$ is selSSR if and only if $X$ satisfies $\mathbf{SV}_R(\Pi_V(\mathbb{CS}(X)),\Pi_V(\mathbb{CS}(X)))$;
    \item $(\mathbb{F}(X),\mathbf{V})$ is selSSR if and only if $X$ satisfies $\mathbf{SV}_R(\Pi_V(\mathbb{F}(X)),\Pi_V(\mathbb{F}(X)))$.
\end{enumerate}
\end{corollary}

Now, let us deal with the absolute versions of star selection principles for Menger and Rothberger cases. We start giving the following principle that will allow us to characterize the absolutely strongly star-Menger property on different hyperspaces with the Vietoris topology.

\begin{definition}\label{AVMdefinition}
Let $X$ be a topological space. We define:

\medskip

\noindent $\mathbf{AV}_{M}(\Pi_V(\Delta),\Pi_V(\Delta))$: For each sequence $\{J_n:n\in\mathbb{N}\}\subseteq\Pi_V(\Delta)$ and each $\mathcal{C}\in\mathbb{C}_V(\Delta)$, there is a sequence $\{\mathcal{V}_n:n\in\mathbb{N}\}\subseteq[\mathcal{C}]^{<\omega}$ such that $\mathcal{J}=\bigcup_{n\in\mathbb{N}}\{  (V_{1,s}^n,\ldots,V_{m_s,s}^n)\in J_n:\text{ there exists } V\in \mathcal{V}_n \text{ such that }\bigcap_{i=1}^{m_s}(V_{i,s}^n)^c\subseteq V, \  V_{i,s}^n\nsubseteq V \  (1\leq i\leq m_s)\}$
is an element of $\Pi_V(\Delta)$.
\end{definition}

By doing some light modifications to the proof of Theorem \ref{VselSSMtheorem}, one can easily prove the following result.

\begin{theorem}\label{VaSSMtheorem}
Given a topological space $X$, the following conditions are equivalent:
\begin{enumerate}
\item[(1)] $(\Delta,\mathbf{V})$ is aSSM;
\item[(2)] $X$ satisfies $\mathbf{AV}_{M}(\Pi_V(\Delta),\Pi_V(\Delta))$.
\end{enumerate}
\end{theorem}

Again, some immediate consequences of the previous theorem are obtained by taking different choices of our family $\Delta$.

        \begin{corollary}\label{VaSSMcorollary}
    Let $X$ be a topological space. Then:
    \begin{enumerate}
        \item $(CL(X),\mathbf{V})$ is aSSM if and only if $X$ satisfies $\mathbf{AV}_M(\Pi_V,\Pi_V)$;
        \item $(\mathbb{K}(X),\mathbf{V})$ is aSSM if and only if $X$ satisfies $\mathbf{AV}_M(\Pi_V(\mathbb{K}(X)),\Pi_V(\mathbb{K}(X)))$;
        \item $(\mathbb{CS}(X),\mathbf{V})$ is aSSM if and only if $X$ satisfies $\mathbf{AV}_M(\Pi_V(\mathbb{CS}(X)),\Pi_V(\mathbb{CS}(X)))$;
        \item $(\mathbb{F}(X),\mathbf{V})$ is aSSM if and only if $X$ satisfies $\mathbf{AV}_M(\Pi_V(\mathbb{F}(X)),\Pi_V(\mathbb{F}(X)))$.
    \end{enumerate}
    \end{corollary}
    
Next principle allows us to characterize the absolutely strongly star-Rothberger property on hyperspaces considering the Vietoris topology.

\begin{definition}\label{AVRdefinition}
Let $X$ be a topological space. We define:

\medskip

\noindent $\mathbf{AV}_{R}(\Pi_V(\Delta),\Pi_V(\Delta))$: For each sequence $\{J_n:n\in\mathbb{N}\}\subseteq\Pi_V(\Delta)$ and each $\mathcal{C}\in\mathbb{C}_V(\Delta)$, there is a sequence $\{C_n:n\in\mathbb{N}\}\subseteq\mathcal{C}$, such that $\mathcal{J}=\bigcup_{n\in\mathbb{N}}\{  (V_{1,s}^n,\ldots,V_{m_s,s}^n)$ $\in J_n: \bigcap_{i=1}^{m_s}(V_{i,s}^n)^c\subseteq C_n, \   V_{i,s}^n\nsubseteq C_n \  (1\leq i\leq m_s)\}$ is an element of $\Pi_V(\Delta)$.
\end{definition}

Following same idea as in the proof of Theorem \ref{VselSSRtheorem}, it is easy to prove the following result.

\begin{theorem}\label{VaSSRtheorem}
Given a topological space $X$, the following conditions are equivalent:
\begin{enumerate}
\item[(1)] $(\Delta,\mathbf{V})$ is aSSR;
\item[(2)] $X$ satisfies $\mathbf{AV}_{R}(\Pi_V(\Delta),\Pi_V(\Delta))$.
\end{enumerate}
\end{theorem}

By taking different choices of the family $\Delta$, we obtain the following particular cases.

\begin{corollary}\label{VaSSRcorollary}
Let $X$ be a topological space. Then:
    \begin{enumerate}
    \item $(CL(X),\mathbf{V})$ is aSSR if and only if $X$ satisfies $\mathbf{AV}_R(\Pi_V,\Pi_V)$;
    \item $(\mathbb{K}(X),\mathbf{V})$ is aSSR if and only if $X$ satisfies $\mathbf{AV}_R(\Pi_V(\mathbb{K}(X)),\Pi_V(\mathbb{K}(X)))$;
    \item $(\mathbb{CS}(X),\mathbf{V})$ is aSSR if and only if $X$ satisfies $\mathbf{AV}_R(\Pi_V(\mathbb{CS}(X)),\Pi_V(\mathbb{CS}(X)))$;
    \item $(\mathbb{F}(X),\mathbf{V})$ is aSSR if and only if $X$ satisfies $\mathbf{AV}_R(\Pi_V(\mathbb{F}(X)),\Pi_V(\mathbb{F}(X)))$.
    \end{enumerate}
\end{corollary}

\section{Absolute and selective versions on hyperspaces with the Fell topology}\label{Fell section}

In this section we now introduce some other technical principles which will help us to make some characterizations of same variations considered in previous section on different hyperspaces with the Fell topology. For that end, we will show a couple of lemmas which will be often used in the proofs of theorems in this section; it is worth to mention that these lemmas are the analogous ones to the Lemmas \ref{cvdeltas are dense sets} and \ref{piVdeltas are open covers} mentioned in Section \ref{Vietoris section}. 

Let us first recall the definition of a $\pi_F(\Delta)$-network of a space $X$ (see \cite{xiy}). Let $\xi$ denote the family
$$\xi=\{(K;V_1,\ldots,V_n): K\text{ is a compact subset of }X, \  V_1,\ldots, V_n\text{  are }$$
$$\text{ open subsets of }X\text{  with  }V_i \cap K^{c} \neq \emptyset \  (1\leq i \leq n)\text{ , }n\in\mathbb{N}\}.\\$$

\begin{definition}[\cite{xiy}]\label{definition piFdelta}
Let $(X,\tau)$ be a topological space. A family $\xi$ is called a $\pi_F(\Delta)$-network of $X$, if for each $U\in \Delta^c$, there exist a $(K; V_1,\ldots, V_n)\in\xi$ and a finite set $F$ with $F\cap V_i\neq\emptyset$ $(1\leq i\leq n)$ such that $K\subset U$ and $F\cap U=\emptyset$. The family of all $\pi_F(\Delta)$-network of $X$ is denoted by $\Pi_F(\Delta)$.
\end{definition}

Similarly to Definition \ref{definition c_V(Delta)-cover}, we can define the notion of $k_F(\Delta)$-cover of a space $X$ by doing some light modifications to the notion of $k_F$-cover introduced in \cite{Z}.

\begin{definition}\label{definition k_F(Delta)-cover}
Let $(X,\tau)$ be a topological space. A family $\mathcal{U}\subseteq \Delta^c$ is called a \emph{$k_F(\Delta)$-cover of $X$}, if for any compact subset $K$ of $X$ and open subsets $V_1,\ldots , V_m$ of $X$, there exists $U\in\mathcal{U}$ and $F\in[X]^{<\omega}$ with $F\cap V_i\neq\emptyset$ ($1 \leq i \leq m$) such that $K\subseteq U$ and $F\cap U=\emptyset$. The family of all $k_F(\Delta)$-covers of $X$ is denoted by $\mathbb{K}_F(\Delta)$.
\end{definition}

\vspace{.3cm}

The following lemma says how $k_F(\Delta)$-covers on a space $X$ can be viewed as dense subspaces of certain hyperspaces of $X$ with the Fell topology.

\begin{lemma}\label{kFdeltas are dense sets}
Let $X$ be a topological space and $\mathcal{U}\subseteq \Delta^c$. Then $\mathcal{U}$ is a $k_F(\Delta)$-cover of $X$ if and only if $\mathcal{U}^c$ is a dense subset of $(\Delta,\mathbf{F})$. 
\end{lemma}
\begin{proof}
$[\rightarrow]$ Let $(\bigcap_{i=1}^nV_i^-)\cap(K^c)^+$ be a nonempty basic open set in $(\Delta,\mathbf{F})$. Since $V_1,\ldots, V_n$ are open subsets of $X$, $K$ is a compact subset of $X$ and $\mathcal{U}$ is a $k_F(\Delta)$-cover, then there exist $U\in\mathcal{U}$ and $F\in[X]^{<\omega}$ with $F\cap V_i\neq\emptyset$ ($1 \leq i \leq m$) such that $K\subseteq U$ and $F\cap U=\emptyset$. Put $A=U^c\in \Delta$. Then, it easy to show that $A$ belongs to $\left[(\bigcap_{i=1}^nV_i^-)\cap(K^c)^+\right]\cap \mathcal{U}^c$. Thus, $\mathcal{U}^c$ is dense set in $(\Delta,\mathbf{F})$.

$[\leftarrow]$ Let $K$ be a compact subset of $X$ and let $V_1,\ldots , V_m$ be open subsets of $X$. Using these sets to define the basic open set $(\bigcap_{i=1}^nV_i^-)\cap(K^c)^+$ of $(\Delta,\mathbf{F})$, we have that $\left[(\bigcap_{i=1}^nV_i^-)\cap(K^c)^+\right]\cap \mathcal{U}^c\neq\emptyset$. So, we take $A\in \left[(\bigcap_{i=1}^nV_i^-)\cap(K^c)^+\right]\cap \mathcal{U}^c$. Using this fact, it is easy to define a finite subset of $X$ so that $F\cap V_i\neq\emptyset$ ($1 \leq i \leq m$). Moreover, if we define $U=A^c$, then $K\subseteq U$ and $F$ can be defined so that $F\cap U=\emptyset$. Thus, $\mathcal{U}$ is a $k_F(\Delta)$-cover of $X$. 
\end{proof}

The next lemma says how $\pi_F(\Delta)$-networks of a space $X$ can be interpreted as open covers of certain hyperspaces of $X$ considering the Fell topology. We remark that both directions in the following lemma were used in some proofs of theorems in \cite{xiy}; for an easier reference of the reader, we explicitly state this fact and give the proof here.

\begin{lemma}\label{piFdeltas are open covers}
Let $X$ be a topological space and $\xi = \{ (K;V_1,\ldots,V_n): \  K\text{ is a }
\newline
\text{compact subset of } X, \  V_1,\ldots, V_n\text{  are open subsets of }X\text{  with  }V_i \cap K^{c} \neq \emptyset \ (1\leq i \leq n)\text{, }n\in\mathbb{N}\}$. Then $\xi$ is a $\pi_F(\Delta)$-network of $X$ if and only if the collection 
$\mathcal{U}=\{(\bigcap_{i=1}^{n}V_i^-)\cap(K^c)^+:(K;V_1,\ldots,V_n) \in \xi\}$ is an open cover of $(\Delta,\mathbf{F})$. 
\end{lemma}
\begin{proof}
$[\rightarrow]$ Let $B\in\Delta$. Since $\xi$ is a $\pi_F(\Delta)$-network of $X$ and $B^c \in \Delta^c$, then there exists $(K;V_1,\ldots,V_n)\in\xi$ and a finite subset $F$ of $X$ with $F\cap V_i\neq\emptyset$ $(1\leq i\leq n)$ such that $K\subseteq B^c$ and $F\cap B^c=\emptyset$. These conditions imply that $B\subseteq K^c$ and $F\subseteq B$. Since $F\cap V_i\neq\emptyset$ for each $i\in\{1,\dots, n\}$, then $B\cap V_i\neq\emptyset$ for each $i\in\{1,\dots, n\}$. It follows that $B\in (\bigcap_{i=1}^{n}V_i^-)\cap(K^c)^+$ with $(\bigcap_{i=1}^{n}V_i^-)\cap(K^c)^+$ being an element of $\mathcal{U}$. Hence, the collection $\mathcal{U}$ is an open cover of $(\Delta,\mathbf{F})$.

$[\leftarrow]$ Let $U\in\Delta^c$. Since $U^c\in\Delta$ and the collection $\mathcal{U}=\{(\bigcap_{i=1}^{n}V_i^-)\cap(K^c)^+:(K;V_1,\ldots,V_n) \in \xi\}$ is an open cover of $(\Delta,\mathbf{F})$, there exists $(\bigcap_{i=1}^{n}V_i^-)\cap(K^c)^+\in\mathcal{U}$ such that $U^c \in (\bigcap_{i=1}^{n}V_i^-)\cap(K^c)^+$. From this fact, we obtain that $K \subseteq U$ and $U^c\cap V_i\neq \emptyset$ for each $i\in\{1\ldots,n\}$. Hence, we can take $x_i\in U^c\cap V_i$ for each $i\in\{1,\ldots,n\}$ and define $F=\{x_i:1\leq i\leq n\}$. Thus, $F$ is a finite subset of $X$ with $F\cap U=\emptyset$ and $F\cap V_i\neq\emptyset$ $(1\leq i\leq n)$. Since the respective element $(K;V_1,\ldots,V_n)$ is an element of $\xi$, we conclude that $\xi$ is a $\pi_F(\Delta)$-network of $X$.
\end{proof}

The following selection principle will help us to characterize the selectively strongly star-Menger property on hyperspaces with the Fell topology.

\begin{definition}\label{SFMdefinition}
Let $X$ be a topological space. We define:

\medskip

\noindent $\mathbf{SF}_{M}(\Pi_F(\Delta),\Pi_F(\Delta))$: For each sequence $\{J_n:n\in\mathbb{N}\}\subseteq\Pi_F(\Delta)$ and each sequence $\{\mathcal{K}_n:n\in\mathbb{N}\}\subseteq \mathbb{K}_F(\Delta)$, there are finite subsets $\mathcal{W}_n\subseteq\mathcal{K}_n$, $n\in\mathbb{N}$, such that $\mathcal{J}=\bigcup_{n\in\mathbb{N}}\{  (K_s^n;V_{1,s}^n,\ldots,V_{m_s,s}^n)\in J_n:\text{ there exists } W \in \mathcal{W}_n \text{ such that } K_s^n$ $\subseteq W, \  V_{i,s}^n\nsubseteq  W \  (1\leq i\leq m_s)\}$
is an element of $\Pi_F(\Delta)$.
\end{definition}

\begin{theorem}\label{FselSSMtheorem}
Given a topological space $X$, the following conditions are equivalent:
\begin{enumerate}
\item[(1)] $(\Delta,\mathbf{F})$ is selSSM;
\item[(2)] $X$ satisfies $\mathbf{SF}_{M}(\Pi_F(\Delta),\Pi_F(\Delta))$.
\end{enumerate}
\end{theorem}
\begin{proof}
$(1) \Rightarrow (2)$: Let $\{J_n:n\in\mathbb{N}\}$ be a sequence of $\pi_F(\Delta)$-networks of $X$ and let $\{\mathcal{K}_n: n\in\mathbb{N}\}$ be a sequence of $k_F(\Delta)$-covers of $X$. For each $n\in\omega$, we denote $J_n=\{(K_s^n; V_{1,s}^n,\ldots,V_{m_s,s}^n):s\in S_n\}$. Applying Lemma \ref{piFdeltas are open covers}, we obtain that, for each $n\in\omega$, the collection $\mathcal{U}_n=\{(\bigcap_{i=1}^{m_s}(V_{i,s}^n)^-)\cap((K_s^n)^c)^+:s\in S_n\}$ is an open cover of $(\Delta, \mathbf{F})$. Furthermore,  by Lemma \ref{kFdeltas are dense sets}, the collection $\mathcal{D}_n=\mathcal{K}_n^c$ is a dense subset of $(\Delta,\mathbf{F})$, for each $n\in\mathbb{N}$.

Thus, applying (1) to the sequence $\{\mathcal{U}_n: n\in\mathbb{N}\}$ and the sequence $\{\mathcal{D}_n:n\in\mathbb{N}\}$, there is a sequence $\{\mathcal{F}_n:n\in\mathbb{N}\}$ of finite sets such that, for each $n\in \mathbb{N}$, $\mathcal{F}_n\subseteq \mathcal{D}_n$ and the collection $\{St(\mathcal{F}_n,\mathcal{U}_n):n\in\mathbb{N}\}$ is an open cover of $(\Delta,\mathbf{F})$. We put, for each $n\in\mathbb{N}$, $\mathcal{W}_n=\{F^c: F\in \mathcal{F}_n\}$. Then, for each $n\in\mathbb{N}$, $\mathcal{W}_n$ is a finite subset of $\mathcal{K}_n$.

Let us show that the collection $\mathcal{J}=\bigcup_{n\in\mathbb{N}}\{(K_s^n;V_{1,s}^n,\ldots,V_{m_s,s}^n)\in J_n: \text{there exists } \ W\in \mathcal{W}_n \text{ such that }K_s^n\subseteq W, \  V_{i,s}^n\nsubseteq  W \  (1\leq i\leq m_s)\}$ is a $\pi_F(\Delta)$-network of $X$.

Let $U\in \Delta^c$. Then $U^c\in \Delta$ and therefore, there exists $n_0\in\omega$ such that $U^c\in St(\mathcal{F}_{n_0}, \mathcal{U}_{n_0})$. 
Thus, there is $(\bigcap_{i=1}^{m_{s_0}}(V_{i,s_0}^{n_0})^-)\cap((K_{s_0}^{n_0})^c)^+\in \mathcal{U}_{n_0}$ such that $\left[(\bigcap_{i=1}^{m_{s_0}}(V_{i,s_0}^{n_0})^-)\cap((K_{s_0}^{n_0})^c)^+\right]\bigcap \mathcal{F}_{n_0}\neq\emptyset$ and $U^c\in (\bigcap_{i=1}^{m_{s_0}}(V_{i,s_0}^{n_0})^-)\cap((K_{s_0}^{n_0})^c)^+$. We fix $F_0\in\left[(\bigcap_{i=1}^{m_{s_0}}(V_{i,s_0}^{n_0})^-)\cap((K_{s_0}^{n_0})^c)^+\right]\bigcap \mathcal{F}_{n_0}$ and define $W_0=F^c_0$. Thus $W_0\in\mathcal{W}_{n_0}$ and, observe that $K_{s_0}^{n_0}\subseteq W_0$ and $V_{i,s_0}^{n_0}\nsubseteq W_0 (1\leq i\leq m_{s_0})$ are obtained from the fact that $F_0\in(\bigcap_{i=1}^{m_{s_0}}(V_{i,s_0}^{n_0})^-)\cap((K_{s_0}^{n_0})^c)^+$. It follows that the corresponding element $(K_{s_0}^{n_0};V_{1,s_0}^{n_0},\ldots,V_{m_{s_0},s_0}^{n_0})$ belongs to $\mathcal{J}$. In addition, from the fact that $U^c\in \bigcap_{i=1}^{m_{s_0}}(V_{i,s_0}^{n_0})^-$, we can fix $x_i\in U^c\cap V_{i,s_0}^{n_0}$ ($i\in\{1,\ldots, m_{s_0}\}$) and let $M=\{x_i : i\in\{1,\ldots ,m_{s_0}\}\}$. So, $M$ is a finite subset of $X$ that satisfies $M\cap V_{i,s_0}^{n_0}\neq\emptyset$ ($1\leq i \leq m_{s_0}$) and $M\cap U=\emptyset$. 
On the other hand, $K_{s_0}^{n_0}\subseteq U$ follows from $U^c\in((K_{s_0}^{n_0})^c)^+$. We conclude that $\mathcal{J}\in\Pi_F(\Delta)$.

$(2) \Rightarrow (1)$: Let $\{\mathcal{U}_n:n\in\mathbb{N}\}$ be a sequence of open covers of $(\Delta,\mathbf{F})$ and let $\{\mathcal{D}_n:n\in\mathbb{N}\}$ be a sequence of dense subsets of $(\Delta,\mathbf{F})$. Without loss of generality, suppose that the elements of each open cover $\mathcal{U}_n$ are basic open subsets in $CL(X)$. Therefore, let us denote, for each $n\in\mathbb{N}$, $\mathcal{U}_n=\{(\bigcap_{i=1}^{m_s}(V_{i,s}^n)^-)\cap((K_s^n)^c)^+:s\in S_n\}$, where each $V_{i,s}^n$ is an open subset of $X$ and $K_s^n$ is a compact subset of $X$ (for every $n\in\mathbb{N}$, $s\in S_n$ and $i\in\{1,\ldots ,m_s\}$).
For each $n\in\mathbb{N}$, let $J_n=\{(K_s^n; V_{1,s}^n,\ldots,V_{m_s,s}^n):s\in S_n\}$. Then, by Lemma \ref{piFdeltas are open covers}, we get that each $J_n$ is a $\pi_F(\Delta)$-network of $X$. Moreover, if we define, for each $n\in\mathbb{N}$, $\mathcal{K}_n = \mathcal{D}_n^c$, then each $\mathcal{K}_n$ is a $k_F(\Delta)$-cover of $X$ by Lemma \ref{kFdeltas are dense sets}.

Applying $(2)$ to the sequence of $\pi_F(\Delta)$-networks $\{J_n:n\in \mathbb{N}\}$ and the sequence of $k_F(\Delta)$-covers $\{\mathcal{K}_n:n\in\mathbb{N}\}$ we obtain a sequence $\{\mathcal{W}_n:n\in\mathbb{N}\}$ with $\mathcal{W}_n$ a finite subset of $\mathcal{K}_n$, for each $n\in\mathbb{N}$, and so that the collection  $\mathcal{J}=\bigcup_{n\in\mathbb{N}}\{  (K_s^n;V_{1,s}^n,\ldots,V_{m_s,s}^n)\in J_n:\text{ there exists } W \in \mathcal{W}_n \text{ such that }K_s^n\subseteq W, \  V_{i,s}^n\nsubseteq  W \  (1\leq i\leq m_s)\}$ is a $\pi_F(\Delta)$-network of $X$. We define, for each $n\in\omega$, $\mathcal{F}_n=\mathcal{W}^c_{n}$. Then $\mathcal{F}_n$ is a finite subset of $\mathcal{D}_n$, for each $n\in\mathbb{N}$.

Let us show that the collection $\{St(\mathcal{F}_n,\mathcal{U}_n):n\in\mathbb{N}\}$ is an open cover of $(\Delta,\mathbf{F})$. Let $A\in\Delta$. Using the fact that $\mathcal{J}$ is a $\pi_F(\Delta)$-network of $X$ with the set $A^c\in\Delta^c$, we obtain an element $(K_{s_0}^{n_0};V_{1,s_0}^{n_0},\ldots,V_{m_{s_0},s_0}^{n_0})$ of $\mathcal{J}$ (for some $n_0\in\mathbb{N}$ and some $s_0\in S_{n_0}$) and a finite set $F\subseteq X$ such that for every $i\in\{1,\ldots ,m_{s_0}\}$, $F\cap V_{i,s_0}^{n_0}\neq\emptyset$, \   $K_{s_0}^{n_0}\subseteq A^c$ and $F\cap A^c=\emptyset$. It easily follows from these last conditions that $A$ belongs to $(\bigcap_{i=1}^{m_{s_0}}(V_{i,s_0}^{n_0})^-)\cap((K_{s_0}^{n_0})^c)^+$. On the other hand, since $(K_{s_0}^{n_0};V_{1,s_0}^{n_0},\ldots,V_{m_{s_0},s_0}^{n_0})\in \mathcal{J}$, there exists $W_0\in \mathcal{W}_{n_0}$ such that $K_{s_0}^{n_0}\subseteq W_0$ and for each $i\in\{1,\ldots ,m_{s_0}\}$, $V_{i,s_0}^{n_0}\nsubseteq W_0$. These conditions imply that the element $F_0 = W_0^c$ belongs to $(\bigcap_{i=1}^{m_{s_0}}(V_{i,s_0}^{n_0})^-)\cap((K_{s_0}^{n_0})^c)^+$. Notice that $F_0\in \mathcal{F}_{n_0}$ and $(\bigcap_{i=1}^{m_{s_0}}(V_{i,s_0}^{n_0})^-)\cap((K_{s_0}^{n_0})^c)^+\in\mathcal{U}_{n_0}$. It means that $A\in St(\mathcal{F}_{n_0},\mathcal{U}_{n_0})$. Therefore, $\{St(\mathcal{F}_n,\mathcal{U}_n):n\in\mathbb{N}\}$ is an open cover of $(\Delta,\mathbf{F})$.
\end{proof}

As a consequence of previous theorem, we can characterize the selectively strongly star-Menger property on different hyperspaces considered with the Fell topology.

        \begin{corollary}\label{FselSSMcorollary}
    Let $X$ be a topological space. Then:
    \begin{enumerate}
        \item $(CL(X),\mathbf{F})$ is selSSM if and only if $X$ satisfies $\mathbf{SF}_M(\Pi_F,\Pi_F)$;
        \item $(\mathbb{K}(X),\mathbf{F})$ is selSSM if and only if $X$ satisfies $\mathbf{SF}_M(\Pi_F(\mathbb{K}(X)),\Pi_F(\mathbb{K}(X)))$;
        \item $(\mathbb{CS}(X),\mathbf{F})$ is selSSM if and only if $X$ satisfies $\mathbf{SF}_M(\Pi_F(\mathbb{CS}(X)),\Pi_F(\mathbb{CS}(X)))$;
        \item $(\mathbb{F}(X),\mathbf{F})$ is selSSM if and only if $X$ satisfies $\mathbf{SF}_M(\Pi_F(\mathbb{F}(X)),\Pi_F(\mathbb{F}(X)))$.
    \end{enumerate}
    \end{corollary}
    
For the Rothberger case we define the following selection principle to get, in the same way, characterizations of the selectively strongly star-Rothberger property on hyperspaces with the Fell topology.

\begin{definition}\label{SFRdefinition}
Let $X$ be a topological space. We define:

\medskip

\noindent $\mathbf{SF}_{R}(\Pi_F(\Delta),\Pi_F(\Delta))$: For each sequence $\{J_n:n\in\omega\}\subseteq\Pi_F(\Delta)$ and each sequence $\{\mathcal{K}_n:n\in\omega\}\subseteq \mathbb{K}_F(\Delta)$, there is a sequence $\{K_n:n\in\omega\}$ with $K_n\in\mathcal{K}_n$, $n\in\omega$, such that $\mathcal{J}=\bigcup_{n\in\mathbb{N}}\{  (K_s^n;V_{1,s}^n,\ldots,V_{m_s,s}^n)\in J_n:K_s^n\subseteq K_n, \  V_{i,s}^n\nsubseteq K_n \  (1\leq i\leq m_s)\}$
is an element of $\Pi_F(\Delta)$.
\end{definition}

\begin{theorem}\label{FselSSRtheorem}
Given a topological space $X$, the following conditions are equivalent:
\begin{enumerate}
\item[(1)] $(\Delta,\mathbf{F})$ is selSSR;
\item[(2)] $X$ satisfies $\mathbf{SF}_{R}(\Pi_F(\Delta),\Pi_F(\Delta))$.
\end{enumerate}
\end{theorem}
\begin{proof}
$(1) \Rightarrow (2)$: By using Lemmas \ref{kFdeltas are dense sets} and \ref{piFdeltas are open covers} and same idea as in the first part of $(1) \Rightarrow (2)$ in the proof of Theorem \ref{FselSSMtheorem}, we can obtain a sequence $\{D_n:n\in\mathbb{N}\}$ of points in $(\Delta,\mathbf{F})$ such that, for each $n\in \mathbb{N}$, $D_n\in \mathcal{D}_n$ and the collection $\{St(D_n,\mathcal{U}_n):n\in\mathbb{N}\}$ is an open cover of $(\Delta,\mathbf{F})$. For each $n\in\mathbb{N}$, we define $K_n=D_n^c$. Then, $K_n\in\mathcal{K}_n$, for each $n\in\mathbb{N}$.

Considering the sequence $\{K_n:n\in\mathbb{N}\}$, let us show that the collection $\mathcal{J}=\bigcup_{n\in\mathbb{N}}\{(K_s^n;V_{1,s}^n,\ldots,V_{m_s,s}^n)\in J_n: K_s^n \subseteq K_n, \  V_{i,s}^n\nsubseteq K_n \ (1\leq i\leq m_s)\}$ is a $\pi_F(\Delta)$-network of $X$.

Let $U\in \Delta^c$. Then, for the element $U^c\in \Delta$, there exists $n_0\in\omega$ such that $U^c\in St(D_{n_0}, \mathcal{U}_{n_0})$. Thus, there is $(\bigcap_{i=1}^{m_{s_0}}(V_{i,s_0}^{n_0})^-)\cap((K_{s_0}^{n_0})^c)^+\in \mathcal{U}_{n_0}$ such that $U^c, D_{n_0} \in (\bigcap_{i=1}^{m_{s_0}}(V_{i,s_0}^{n_0})^-)\cap((K_{s_0}^{n_0})^c)^+$.

As $K_{n_0}=D_{n_0}^c$ and $D_{n_0}\in (\bigcap_{i=1}^{m_{s_0}}(V_{i,s_0}^{n_0})^-)\cap((K_{s_0}^{n_0})^c)^+$, it easily follows that $K_{s_0}^{n_0}\subseteq K_{n_0}$ and $ \ V_{i,s_0}^{n_0}\nsubseteq K_{n_0} (1\leq i\leq m_{s_0})$. This implies that the corresponding $(K_{s_0}^{n_0};V_{1,s_0}^{n_0},\ldots,V_{m_{s_0},s_0}^{n_0})$ is an element of $\mathcal{J}$. Furthermore, since $U^c$ belongs to $(\bigcap_{i=1}^{m_{s_0}}(V_{i,s_0}^{n_0})^-)\cap((K_{s_0}^{n_0})^c)^+$, we get that $K_{s_0}^{n_0}\subseteq U$ and also, we can define a finite subset $F$ of $X$ such that $F\cap V_{i,s_0}^{n_0}\neq\emptyset$ and $F\cap U=\emptyset$. We conclude that $\mathcal{J}\in\Pi_F(\Delta)$.

$(2) \Rightarrow (1)$: Following same arguments and employing Lemmas \ref{kFdeltas are dense sets} and \ref{piFdeltas are open covers} as in the first part of $(2) \Rightarrow (1)$ in the proof of Theorem \ref{FselSSMtheorem}, we obtain a sequence $\{K_n:n\in\mathbb{N}\}$ such that, for each $n\in\mathbb{N}$, $K_n\in\mathcal{K}_n$, and the collection  $\mathcal{J}=\bigcup_{n\in\mathbb{N}}\{(K_s^n;V_{1,s}^n,\ldots,V_{m_s,s}^n)\in J_n: K_s^n\subseteq K_n, \  V_{i,s}^n\nsubseteq K_n \  (1\leq i\leq m_s)\}$ is a $\pi_F(\Delta)$-network of $X$. We define, for each $n\in\mathbb{N}$, $D_n= K^c_{n}$. Then, $\{D_n:n\in\mathbb{N}\}$ is a sequence of points in $(\Delta,\mathbf{F})$ with $D_n\in\mathcal{D}_n$, for each $n\in\mathbb{N}$.

We claim that the collection $\{St(D_n, \mathcal{U}_n):n\in\mathbb{N}\}$ is an open cover of $(\Delta,\mathbf{F})$. Indeed, let $A\in\Delta$. Since $\mathcal{J}$ is a $\pi_F(\Delta)$-network of $X$ and $A^c\in\Delta^c$, there exists $(K_{s_0}^{n_0};V_{1,s_0}^{n_0},\ldots,V_{m_{s_0},s_0}^{n_0}) \in \mathcal{J}$ (for some $n_0\in\mathbb{N}$ and some $s_0\in S_{n_0}$) and a finite set $F\subseteq X$ such that $F\cap V_{i,s_0}^{n_0}\neq\emptyset$ ($1\leq i \leq m_{s_0}$), \   $K_{s_0}^{n_0}\subseteq A^c$ and $F\cap A^c=\emptyset$. The fact $(K_{s_0}^{n_0};V_{1,s_0}^{n_0},\ldots,V_{m_{s_0},s_0}^{n_0}) \in \mathcal{J}$ means that $K_{s_0}^{n_0}\subseteq K_{n_0}$ and $V_{i,s_0}^{n_0}\nsubseteq K_{n_0}$ ( $1\leq i \leq m_{s_0}$) and hence, $D_{n_0}\in(\bigcap_{i=1}^{m_{s_0}}(V_{i,s_0}^{n_0})^-)\cap((K_{s_0}^{n_0})^c)^+$. Moreover, the conditions $F\cap V_{i,s_0}^{n_0}\neq\emptyset$ ($1\leq i \leq m_{s_0}$), \   $K_{s_0}^{n_0}\subseteq A^c$ and $F\cap A^c=\emptyset$ mean that $A\in (\bigcap_{i=1}^{m_{s_0}}(V_{i,s_0}^{n_0})^-)\cap((K_{s_0}^{n_0})^c)^+$. Thus, $A\in St(D_{n_0},\mathcal{U}_{n_0})$. In conclusion, $\{St(D_n,\mathcal{U}_n):n\in\mathbb{N}\}$ is an open cover of $(\Delta,\mathbf{F})$.
\end{proof}

Now, we can characterize the selectively strongly star-Rothberger property on several hyperspaces with the Fell topology.

        \begin{corollary}\label{FselSSRcorollary}
    Let $X$ be a topological space. Then:
    \begin{enumerate}
        \item $(CL(X),\mathbf{F})$ is selSSR if and only if $X$ satisfies $\mathbf{SF}_R(\Pi_F,\Pi_F)$;
        \item $(\mathbb{K}(X),\mathbf{F})$ is selSSR if and only if $X$ satisfies $\mathbf{SF}_R(\Pi_F(\mathbb{K}(X)),\Pi_F(\mathbb{K}(X)))$;
        \item $(\mathbb{CS}(X),\mathbf{F})$ is selSSR if and only if $X$ satisfies $\mathbf{SF}_R(\Pi_F(\mathbb{CS}(X)),\Pi_F(\mathbb{CS}(X)))$;
        \item $(\mathbb{F}(X),\mathbf{F})$ is selSSR if and only if $X$ satisfies $\mathbf{SF}_R(\Pi_F(\mathbb{F}(X)),\Pi_F(\mathbb{F}(X)))$.
    \end{enumerate}
    \end{corollary}

Now we characterize the absolute versions of the Menger-type and Rothberger-type star selection principles with the Fell topology. The following principle allows us to characterize the absolutely strongly star-Menger property on several hyperspaces with the Fell topology.

\begin{definition}\label{AFMdefinition}
Let $X$ be a topological space. We define:

\medskip

\noindent $\mathbf{AF}_{M}(\Pi_F(\Delta),\Pi_F(\Delta))$: For each sequence $\{J_n:n\in\omega\}\subseteq\Pi_F(\Delta)$ and each $\mathcal{K}\in\mathbb{K}_F(\Delta)$, there is a sequence $\{\mathcal{W}_n:n\in\omega\}\subseteq[\mathcal{K}]^{<\omega}$ such that $\mathcal{J}=\bigcup_{n\in\mathbb{N}}\{  (K_s^n;V_{1,s}^n,\ldots,V_{m_s,s}^n)\in J_n:\text{ there exists } W\in \mathcal{W}_n \text{ such that }K_s^n\subseteq W, \   V_{i,s}^n$ $\nsubseteq  W  \  (1\leq i\leq m_s)\}$
is an element of $\Pi_F(\Delta)$.
\end{definition}

Adapting the proof of Theorem \ref{FselSSMtheorem}, one can easily obtain the following result.

\begin{theorem}\label{FaSSMtheorem}
Given a topological space $X$, the following conditions are equivalent:
\begin{enumerate}
\item[(1)] $(\Delta,\mathbf{F})$ is aSSM;
\item[(2)] $X$ satisfies $\mathbf{AF}_{M}(\Pi_F(\Delta),\Pi_F(\Delta))$.
\end{enumerate}
\end{theorem}

We obtain the following particular cases by taking different choices of our family $\Delta$.

        \begin{corollary}\label{FaSSMcorollary}
    Let $X$ be a topological space. Then:
    \begin{enumerate}
        \item $(CL(X),\mathbf{F})$ is aSSM if and only if $X$ satisfies $\mathbf{AF}_M(\Pi_F,\Pi_F)$;
        \item $(\mathbb{K}(X),\mathbf{F})$ is aSSM if and only if $X$ satisfies $\mathbf{AF}_M(\Pi_F(\mathbb{K}(X)),\Pi_F(\mathbb{K}(X)))$;
        \item $(\mathbb{CS}(X),\mathbf{F})$ is aSSM if and only if $X$ satisfies $\mathbf{AF}_M(\Pi_F(\mathbb{CS}(X)),\Pi_F(\mathbb{CS}(X)))$;
        \item $(\mathbb{F}(X),\mathbf{F})$ is aSSM if and only if $X$ satisfies $\mathbf{AF}_M(\Pi_F(\mathbb{F}(X)),\Pi_F(\mathbb{F}(X)))$.
    \end{enumerate}
    \end{corollary}
    
The following principle will give us characterizations of the absolutely strongly star-Rothberger property on hyperspaces considering the Fell topology.

\begin{definition}\label{AFRdefinition}
Let $X$ be a topological space. We define:

\medskip

\noindent $\mathbf{AF}_{R}(\Pi_F(\Delta),\Pi_F(\Delta))$: For each sequence $\{J_n:n\in\mathbb{N}\}\subseteq\Pi_F(\Delta)$ and each $\mathcal{K}\in\mathbb{K}_F(\Delta)$, there is a sequence $\{K_n:n\in\mathbb{N}\}\subseteq\mathcal{K}$, such that $\mathcal{J}=\bigcup_{n\in\mathbb{N}}\{  (K_s^n;V_{1,s}^n,\ldots,$ $V_{m_s,s}^n)\in J_n:K_s^n\subseteq K_n, \  V_{i,s}^n\nsubseteq K_n \  (1\leq i\leq m_s)\}$
is an element of $\Pi_F(\Delta)$.
\end{definition}

Similarly to the proof of Theorem \ref{FselSSRtheorem}, one can prove the following result.

\begin{theorem}\label{FaSSRtheorem}
Given a topological space $X$, the following conditions are equivalent:
\begin{enumerate}
\item[(1)] $(\Delta,\mathbf{F})$ is aSSR;
\item[(2)] $X$ satisfies $\mathbf{AF}_{R}(\Pi_F(\Delta),\Pi_F(\Delta))$.
\end{enumerate}
\end{theorem}

The following particular cases are obtained by choosing different families $\Delta$.

        \begin{corollary}\label{FaSSRcorollary}
    Let $X$ be a topological space. Then:
    \begin{enumerate}
        \item $(CL(X),\mathbf{F})$ is aSSR if and only if $X$ satisfies $\mathbf{AF}_R(\Pi_F,\Pi_F)$;
        \item $(\mathbb{K}(X),\mathbf{F})$ is aSSR if and only if $X$ satisfies $\mathbf{AF}_R(\Pi_F(\mathbb{K}(X)),\Pi_F(\mathbb{K}(X)))$;
        \item $(\mathbb{CS}(X),\mathbf{F})$ is aSSR if and only if $X$ satisfies $\mathbf{AF}_R(\Pi_F(\mathbb{CS}(X)),\Pi_F(\mathbb{CS}(X)))$;
        \item $(\mathbb{F}(X),\mathbf{F})$ is aSSR if and only if $X$ satisfies $\mathbf{AF}_R(\Pi_F(\mathbb{F}(X)),\Pi_F(\mathbb{F}(X)))$.
    \end{enumerate}
    \end{corollary}

\textbf{Concluding remark.} It worth remarking that, by doing some slightly modifications to definitions given in Sections \ref{Vietoris section} and \ref{Fell section}, it is possible to easily obtain hit-and-miss generalizations that involve both the Vietoris topology and the Fell topology (similar to how it was done in \cite{CRT2} and \cite{CRT}).  However, for convenience of the reader, the author decided to establish simpler notation results depending of which topology is being considered.

\textsc{Departamento de Matem\'aticas, Facultad de Ciencias, UNAM, Circuito Exterior S/N, Ciudad Universitaria, CP 04510, Ciudad de M\'exico, M\'exico.}\par\nopagebreak

\vspace{.1cm}
\textit{Email address}: J. Casas-de la Rosa: \texttt{olimpico.25@hotmail.com}

\end{document}